\documentclass[11pt,reqno]{amsart}
\usepackage{amsmath,amssymb,geometry,color,tikz-cd}
\usepackage{titletoc}
\usepackage[backref,pagebackref,pdftex,hyperindex]{hyperref}
\usepackage[alphabetic]{amsrefs}
\usepackage{amssymb}
\usepackage{bbm}
\usepackage{enumitem}
\usepackage{upgreek}

\newtheorem{proposition}{Proposition}[section]
\newtheorem{theorem}[proposition]{Theorem}
\newtheorem{lemma}[proposition]{Lemma}

\theoremstyle{definition}
\newtheorem{remark}[proposition]{Remark}
\newtheorem{definition}[proposition]{Definition}
\newtheorem{example}[proposition]{Example}
\newtheorem{question}[proposition]{Question}

\title{Boundedness of log Fano pairs with certain K-stability}
\author{Konstantin Loginov and Chuyu Zhou}
\address{Department of Mathematics, Yonsei University, 50 Yonsei-ro, Seodaemun-gu, Seoul 03722, Republic of Korea}
\email{chuyuzhou1@gmail.com}
\address{Steklov Mathematical Institute of Russian Academy of Sciences, Moscow, Russia}
\email{loginov@mi-ras.ru}
\date{} 

\thanks{
        \newline
		\indent
		\emph{Competing interests}: This work started while the authors were enjoying the SinG school in Geometry in Trento (Italy) during January 23--27, 2023, where the hospitality is gratefully acknowledged.
		}

\newcommand{\ord}{{\rm {ord}}}

\newcommand{\vol}{{\rm {vol}}}

\newcommand{\red}{{\rm {red}}}

\newcommand{\dt}{{\rm {dt}}}

\newcommand{\bA}{\mathbb{A}}

\newcommand{\bC}{\mathbb{C}}

\newcommand{\bN}{\mathbb{N}}

\newcommand{\bP}{\mathbb{P}}
\newcommand{\bQ}{\mathbb{Q}}
\newcommand{\bR}{\mathbb{R}}

\newcommand{\mD}{\mathcal{D}}
\newcommand{\mE}{\mathcal{E}}

\newcommand{\mG}{\mathcal{G}}

\newcommand{\mO}{\mathcal{O}}
\newcommand{\mP}{\mathcal{P}}

\newcommand{\mX}{\mathcal{X}}
\newcommand{\mY}{\mathcal{Y}}

\begin{document}

\begin{abstract}
We prove several boundedness results for log Fano pairs with certain K-stability. In particular, we prove that K-semistable log Fano pairs of Maeda type in each dimension form a log bounded family. We also compute K-semistable domains for some examples.
\end{abstract}

\maketitle

\setcounter{tocdepth}{1}
\tableofcontents

\section{Introduction}

We work over the complex number field $\bC$ throughout the article.

In this paper, we investigate the relation between K-stability and boundedness of log Fano pairs. A log pair $(X, D)$ where $D=\sum_{i=1}^kD_i$ is called a \emph{log Fano manifold} if it is log smooth and $D_i$ are distinct prime divisors on $X$ such that $-K_X-D$ is ample. Log Fano manifolds were classified in dimension $\leq 3$ by H. Maeda, see \cite{Maeda86, Loginov21}. It is well known that log Fano manifolds are not bounded in every dimension starting from $2$. Indeed, considering log Fano manifolds $(\mathbb{F}_n, s)$ where $\mathbb{F}_n$ is the $n$-th Hirzebruch surface and $s$ is a $(-n)$-section, it is clear that they are not bounded. However, if the log Fano manifolds have certain K-stability conditions, they are bounded in dimension $2$. 

Indeed, in \cite{Fujita20}, the following definition was introduced. A \emph{log Fano pair of Maeda type} is a log pair $(X, \sum_{i=1}^kc_iD_i)$ where $(X, \sum_{i=1}^kD_i)$ is a log Fano manifold,  $c_i\in [0, 1)\cap \bQ$, and $-K_X-\sum_{i=1}^kc_iD_i$ is ample. Then, in dimension $2$ there is the following

\begin{theorem}[{\cite[Theorem 1]{Fujita20}}]
\label{thm-2-dim}
Let $(X, \sum_ic_iD_i)$ be a log Fano pair of Maeda type such that $\dim X=2$. Then $(X, \sum_ic_iD_i)$ is K-semistable if and only if it is isomorphic to 
\begin{itemize}
\item
 $(\mathbb{P}^2, aC)$ with $C$ a smooth conic and $0\leq a \leq 3/4$, or
\item
$(\mathbb{P}^1\times \mathbb{P}^1, aC)$ with $C$ the diagonal and $0\leq a \leq 1/2$.
\end{itemize}
\end{theorem}

The same phenomenon appears in dimension three. More precisely, in \cite[Theorem 1]{Loginov21} it is proven that three-dimensional K-semistable log Fano pairs of Maeda type with at least two components in the boundary are log bounded. Then it is natural to ask whether K-semistable log Fano pairs of Meada type in every dimension are bounded. One of the main results of this paper is to confirm this boundedness.

\begin{theorem}\label{thm: main1}{\rm{(= Theorem \ref{thm: bdd maeda} )}}
K-semistable log Fano pairs of Maeda type in dimension~$d$ are log bounded.
\end{theorem}

We also consider another set of log pairs with certain K-stability.
Fix two positive integers $d$ and $k$, a positive number $v$,  and a positive integer $I$. We consider the set $\mE:=\mE(d,k,v, I)$ of log pairs $(X, \sum_{i=1}^k D_i)$ satisfying the following conditions:
\begin{enumerate}
\item $X$ is a Fano variety of dimension $d$ and $(-K_X)^d=v$;
\item $D_i$ is an effective $\bQ$-divisor satisfying $D_i\sim_\bQ -K_X$ for every $1\leq i\leq k$;
\item $I(K_X+D_i)\sim 0$ for every $1\leq i\leq k$;
\item there exists $(c_1,...,c_k)\in \Delta^k$ such that $(X, \sum_i c_iD_i)$ is K-semistable, where $\Delta^k:=\{(c_1,...,c_k)\ |\ \textit{$c_i\in [0,1)\cap \bQ$ and $0\leq \sum_i c_i<1$}\}$.
\end{enumerate}

We confirm the boundedness of $\mE$.

\begin{theorem}\label{thm: main2}{\rm{(= Theorem \ref{thm: bdd II} )}}
The set $\mE$ is log bounded.
\end{theorem}

When there is only one component in the boundary, i.e. $k=1$, the log boundedness of $\mE$ is confirmed in \cite[Section 5]{Zhou23}. Also we note that Theorem \ref{thm: main2} seems related to \cite[Theorem 1.4]{Chen20}. However, the coefficients are fixed in \cite[Theorem 1.4]{Chen20}, while in our setting, the coefficients of $(X, \sum_{i}^k c_iD_i)$ are randomly changed in $[0,1]$.

In the last section, we also compute K-semistable domains (see Definition \ref{def: kss domain}) for various examples. In particular, we describe the domains for two special classes of pairs as follows.

\begin{theorem}{\rm{(Theorem \ref{thm: polytope example}, Example \ref{exa: polytope example})}}
For $n\geq 2$, the K-semistable domain for $(\bP^n, Q+L)$ is a polytope generated by the following three K-semistable log pairs
$$\bP^n, \ \ \  (\bP^n, \frac{n+1}{2n}Q), \ \ \ (\bP^n, \frac{n}{2(n-1)}Q+\frac{1}{n-1}L).$$
Here $Q$ is a smooth quadric hypersurface and $L$ is a hyperplane such that $(\bP^n, Q+L)$ is log smooth. 
\end{theorem}

\begin{theorem}{\rm{(Theorem \ref{thm: polytope example II}, Example \ref{exa: polytope example II})}}
For $n\geq 3$, the K-semistable domain for $(\bP^n, Q+Q')$ is a polytope generated by the following four K-semistable log pairs
$$\bP^n,\ \ \ (\bP^n, \frac{n+1}{2n}Q),\ \ \ (\bP^n, \frac{n+1}{2n}Q'),\ \ \ (\bP^n, \frac{n+1}{2(n-1)}Q+\frac{n+1}{2(n-1)}Q').$$
Here $Q, Q'$ are two distinct smooth quadric hypersurfaces such that $(\bP^n, Q+Q')$ is log smooth. 
\end{theorem}

\noindent
\subsection*{Acknowledgement}
We are grateful to Julia Schneider for her help on Latex techniques.
 C. Zhou is supported by the grant of European Research Council (ERC-804334).
K. Loginov is supported by Russian Science Foundation under grant 21-71-00112.

\section{Preliminaries}

 For the standard definitions of birational geometry, including the concepts of klt, lc, dlt and $\epsilon$-lc singularities, we refer to \cite{KM98,Kollar13, Birkar19}.
 
 We say that $(X,\Delta)$ is a \emph{couple} if $X$ is a normal projective variety and $\Delta$ is an effective $\bQ$-divisor on $X$. We say a couple $(X, \Delta)$ is a \textit{log pair} if $K_X+\Delta$ is $\bQ$-Cartier.  The log pair $(X,\Delta)$ is called \emph{log Fano} if it admits lc singularities and $-(K_X+\Delta)$ is ample; if $\Delta=0$, we just say $X$ is a \emph{Fano variety}. 
 The log pair $(X,\Delta)$ is called a \emph{log Calabi-Yau pair} if $K_X+\Delta\sim_\bQ 0$.

\subsection{K-stability}

Let $(X,\Delta)$ be a log pair. Suppose $f\colon Y\to X$ is a proper birational morphism between normal varieties and $E$ is a prime divisor on $Y$, we say that $E$ is a prime divisor over $X$ and we define the following invariant
$$A_{(X,\Delta)}(E):=1+\ord_E(K_Y-f^*(K_X+\Delta)), $$
which is called the \emph{log discrepancy} of $E$ associated to the log pair $(X,\Delta)$.
If $(X,\Delta)$ is a log Fano pair, we define the following invariant
$$S_{(X,\Delta)}(E):=\frac{1}{\vol(-K_X-\Delta)}\int_0^\infty \vol(-f^*(K_X+\Delta)-xE){\rm{d}}x .$$
Denote $\beta_{(X,\Delta)}(E):=A_{(X,\Delta)}(E)-S_{(X,\Delta)}(E)$ and $\delta(X, \Delta):=\inf_E \frac{A_{(X, \Delta)}(E)}{S_{(X, \Delta)}(E)}$, where $E$ runs over all prime divisors over $X$ for the latter definition. By the works \cite{Fuj19, Li17}, one can define the K-semistability of a log Fano pair as follows.
\begin{definition}
Let $(X,\Delta)$ be a klt log Fano pair. 
We say that $(X,\Delta)$ is \emph{K-semistable} if $\beta_{X,\Delta}(E)\geq 0$ for any prime divisor $E$ over $X$ (or equivalently, $\delta(X, \Delta)\geq 1$).
\end{definition}

\begin{definition}
Let $(X, \Delta)$ be a log Calabi-Yau pair, i.e. $K_X+\Delta\sim_\bQ 0$. We say  $(X,\Delta)$ is \emph{K-semistable} if 
$(X, \Delta)$ is log canonical (this is equivalent to saying that $\beta_{(X, \Delta)}(E)\geq 0$ for any prime divisor $E$ over $X$ since $S_{(X, \Delta)}(E)$=0 in this case, see \cite[Corollary 9.4]{BHJ17}, \cite{Oda13}).
\end{definition}

\begin{remark}\label{rmk: remarks}
Let $(X, \Delta)$ be a klt log Fano pair and $0\leq D\sim_\bQ -K_X-\Delta$. We will use the following results.
\begin{enumerate}
\item For a rational number $0\leq c<1$, we have the following formula by a simple integral computation
$$S_{(X, \Delta+cD)}(E)=(1-c)S_{(X, \Delta)}(E), $$
where $E$ is any prime divisor over $X$.
\item Suppose $(X, \Delta)$ is of dimension $d$, we have the following relationship between alpha invariant and delta invariant by \cite[Theorem A]{BJ20}:
$$\frac{1}{d+1}\cdot\delta(X, \Delta)\leq \alpha(X, \Delta)\leq \delta(X, \Delta). $$
Here $\alpha(X, \Delta)$ is the alpha-invariant of the log Fano pair $(X, \Delta)$ (e.g. \cite{CS08}).
\end{enumerate}
\end{remark}

\subsection{Complements}

\begin{definition}
Let $(X, \Delta)$ be a log Fano pair. We say a $\bQ$-divisor $D\geq 0$ is a \emph{complement} of $(X, \Delta)$ if $(X, \Delta+D)$ is log canonical and $K_X+\Delta+D\sim_\bQ 0$; we say $D$ is an \emph{$N$-complement} for some positive integer $N$ if $D$ is a complement and $N(K_X+\Delta+D)\sim 0$.
\end{definition}

\begin{theorem}{\rm{(\cite{Birkar19})}}\label{thm: complement}
Let $(X, \Delta)$ be a klt log Fano pair, then there exists a positive number $N$ depending only on the dimension of $X$ and the coefficients of $\Delta$ such that $(X, \Delta)$ admits an $N$-complement.
\end{theorem}

\subsection{Boundedness of Fano varieties}

We recall some results on boundedness of Fano varieties.

\begin{definition}
Let $\mP$ be a set of normal projective varieties of dimension $d$. We say $\mP$ is bounded if there exists a projective morphism $ \mY\to T$ between schemes of finite type such that for any $X\in \mP$, there exists a closed point $t\in T$ such that $X\cong \mY_t$. Let $\mP'$ be a set of couples of dimension $d$, we say 
$\mP'$ is log bounded if there exist a projective morphism $\mY\to T$ between schemes of finite type and a reduced Weil divisor $\mD$ on $\mY$ such that for any $(X, \Delta)\in \mP$, there exists a closed point $t\in T$ such that $(X, \red(\Delta))\cong (\mY_t, \mD_t)$. Here $\red(\Delta)$ means taking all the coefficients of components in $\Delta$ to be one. 
\end{definition}

\begin{theorem}{\rm{(\cite{Birkar19, Birkar21})}}\label{thm: BAB}
Fix a positive integer $d$ and a positive real number $\epsilon>0$. Then the following set lies in a bounded family:
$$\{X\ |\ \textit{$X$ is of dimension $d$ and $(X, \Delta)$ is $\epsilon$-lc log Fano for some $\bQ$-divisor $\Delta$ on $X$}\}. $$
\end{theorem}

\begin{theorem}{\rm{(\cite{Jiang20})}}\label{thm: jiang}
Fix a positive integer $d$ and two positive real numbers $v, \alpha_0>0$. Then the following set lies in a bounded family:
$$\{X\ |\ \textit{$X$ is a Fano variety of dimension $d$ with $(-K_X)^d\geq v$ and $\alpha(X)\geq \alpha_0$}\}. $$

\end{theorem}

\section{Boundedness I}

The goal of this section is to prove Theorem \ref{thm: main1}. We recall the following definition introduced in \cite{Fujita 20}:
\begin{definition}
A \emph{log Fano pair of Maeda type} is a pair $(X, \sum_{i=1}^kc_iD_i)$ where $(X, \sum_{i=1}^kD_i)$ is a log Fano manifold,  $c_i\in [0, 1)\cap \bQ$, and $-K_X-\sum_{i=1}^kc_iD_i$ is ample. 
 Here by a \emph{log Fano manifold} we mean a log pair $(X, \sum_{i=1}^kD_i)$ such that it is log smooth and $D_i$ are distinct prime divisors on $X$ with ample $-K_X-\sum_{i=1}^kD_i$. 
\end{definition}

To give an example, the log pair $(\mathbb{P}^2, \frac{1}{2} (L_1 + L_2))$, where $L_1$ and $L_2$ are two different lines on $\mathbb{P}^2$, is of Maeda type, while the log pair $(\mathbb{P}^2, \frac{1}{2} (L + Q))$, where $L$ is a general line and $Q$ is a general conic on $\mathbb{P}^2$, is not of Maeda type.

\begin{definition}
Let $X$ be a projective normal variety and $L$ a pseudo-effective $\bQ$-line bundle on $X$. Let $D$ be an effective $\bQ$-divisor on $X$, then the \emph{pseudo-effective threshold} of $D$ with respect to $L$ is defined as 
$$\tau(D; L):=\sup\{t\in \bR\ |\ \textit{$L-tD$ is pseudo-effective}\}. $$

\end{definition}

\begin{lemma}\label{lem: pe threshold}
Let $(X, D:=\sum_{i=1}^kD_i)$ be a log smooth pair such that $D_i$ are prime divisors on $X$ and $-K_X-\sum_{i=1}^kD_i$ is big and nef, then there exists a positive number $a_d$ depending only on the dimension $d$ such that
$$\tau(D; -K_X-\sum_iD_i)\geq a_d. $$
\end{lemma}

\begin{proof}
Suppose there exists a log pair of Maeda type such that
$$\tau(D; -K_X-\sum_iD_i)< \frac{1}{d+1}. $$
Denote by $M:=-K_X-D$, then we see that $K_X+lM$ is not pseudo-effective for $1\leq l\leq d+1$. This implies the following vanishing for $1\leq l\leq d+1$
$$H^0(X, K_X+lM)=0. $$
On the other hand, by Kawamata-Viehweg vanishing, we have
$$H^i(X, K_X+lM)=0 $$
for any $i\geq 1$ and $l\geq 1$.
Hence we see
$$\chi(X, K_X+lM)=0 $$
for $1\leq l\leq d+1$.
By Hirzebruch-Riemann-Roch formula, the Euler characteristic $\chi(X, K_X+lM)$ is a polynomial of degree $d$ with leading term $\frac{M^d}{d!}l^d$. This leads to a contradiction as the equation $\chi(X, K_X+lM)=0$ cannot have $d+1$ distinct roots.

We are done by taking $a_d=\frac{1}{d+1}.$
\end{proof}

\begin{theorem}\label{thm: bdd maeda}
K-semistable log Fano pairs of Maeda type in dimension $d$ form a log bounded family.
\end{theorem}

\begin{proof}

Let $(X, \sum_{i=1}^kc_iD_i)$ be a K-semistable log pair of Maeda type, we first show that $X$ belongs to a bounded family.  We claim that there exists a positive number $\epsilon_d$ depending only on the dimension $d$ such that $1-c_i\geq \epsilon_d$ for each $i$. We focus on $c_1$.

By Lemma \ref{lem: pe threshold}, there exists a positive number $a_d$ depending only on the dimension $d$ such that
$$\tau(D_1; -K_X-\sum_i c_iD_i) > \tau(D; -K_X-\sum_i D_i)\geq a_d.$$
Thus there exists $0\leq \Delta\sim_\bQ -K_X-\sum_ic_iD_i$ such that $\ord_{D_1}(\Delta)\geq a_d$. Let $\alpha:=\alpha(X, \sum_i c_iD_i)$ be the alpha invariant of the log Fano pair $(X, \sum_ic_iD_i)$, then the log pair $(X, \sum_ic_iD_i+\frac{\alpha}{2}\Delta)$ is log canonical. Thus we see
$$c_1+\frac{\alpha}{2}\cdot a_d \leq 1.$$
Since $(X, \sum_ic_iD_i)$ is K-semistable, then $\alpha\geq \frac{1}{d+1}$ by Remark \ref{rmk: remarks} (2). Taking $\epsilon_d:= \frac{a_d}{2(d+1)}$, we see that
$$1-c_1\geq \epsilon_d. $$
Similarly, this inequality also holds for other $c_i$, which implies that $(X, \sum_ic_iD_i)$ is $\epsilon_d$-lc. By Theorem \ref{thm: BAB}, we see that $X$ belongs to a bounded family.

To derive the log boundedness, we take a very ample line bundle $A_X$ on $X$ such that $A_X^d$ is upper bounded by $M(d)$, which only depends on the dimension $d$. Note that $-K_X\cdot A_X^{d-1}$ is also upper bounded by a positive number $N(d)$, which depends only on the dimension $d$, then we see 
$$\sum_iD_i\cdot A_X^{d-1}< -K_X\cdot A_X^{d-1}\leq N(d) $$
since $-K_X-\sum_iD_i$ is ample. This means that the degree of $\sum_i D_i$ is upper bounded. By a standard argument on Chow scheme we see that $(X, \sum_iD_i)$ belongs to a log bounded family.
\end{proof}

\begin{example}
Note that if a log Fano pair $(X, \sum D_i)$ is not log smooth, then even if $(X, \sum a_iD_i)$ is K-semistable for some numbers $a_i$, it is not true that $(X, \sum D_i)$ is log bounded. 

Indeed, for any $n\geq 1$ put $X=\mathbb{P}(1,1,n)$ with coordinates $x,y,z$, and $D=\{z=0\}$. Note that $D\sim nH$ where $H$ is a generator of the class group of $X$. It is easy to see that the pair $(X, D)$ is a dlt log Fano pair. Then $(X, (1-\frac{1}{n})D)$ is a K-semistable log Fano pair (e.g. \cite[Prop 2.11]{LZ22})  while there is no boundedness.
\end{example}

We conclude by formulating the following question:
\begin{question}
Fix a natural number $d$ and a rational number $\epsilon>0$. Suppose that $(X, D=\sum D_i)$ is a dlt log Fano pair of dimension $d$, and $(X, (1-\epsilon) D)$ is $\epsilon$-lc. Also, assume that $(X, \sum c_i D_i)$ is K-semistable for some rational numbers $c_i\in [0, 1)$. Is it true that $(X, D)$ is log bounded?
\end{question}

\section{Boundedness II}

Fix two positive integers $d$ and $k$, a positive number $v$,  and a positive integer $I$. We consider the set $\mE:=\mE(d,k,v, I)$ of log pairs $(X, \sum_{i=1}^k D_i)$ satisfying the following conditions:
\begin{enumerate}
\item $X$ is a Fano variety of dimension $d$ and $(-K_X)^d=v$;
\item $D_i$ is an effective $\bQ$-divisor satisfying $D_i\sim_\bQ -K_X$ for every $1\leq i\leq k$;
\item $I(K_X+D_i)\sim 0$ for every $1\leq i\leq k$;
\item there exists $(c_1,...,c_k)\in \Delta^k$ such that $(X, \sum_i c_iD_i)$ is K-semistable, where $\Delta^k:=\{(c_1,...,c_k)\ |\ \textit{$c_i\in [0,1)\cap \bQ$ and $0\leq \sum_i c_i<1$}\}$.
\end{enumerate}

\begin{theorem}\label{thm: bdd II}
The set $\mE$ is log bounded.
\end{theorem}

\begin{proof}
We divide the proof into several steps.

\

\textit{Step 1}. In this step, we explain that it is enough to show that the set 
$$\mG:=\{X\ |\ \textit{$(X, \sum_{i=1}^kD_i)\in \mE$}\}$$
is bounded. Suppose $\mG$ is bounded, then for each $(X, \sum_iD_i)\in \mE$, there exists a very ample line bundle $A_X$ on $X$ such that $A_X^d\leq M(d)$ and $-K_X\cdot A_X^{d-1}\leq N(d)$, where $M(d)$ and $N(d)$ are positive numbers depending only on the dimension $d$. For each $i$, it is clear that $D_i\cdot A_X^{d-1}\leq N(d)$. As the coefficients of $D_i$ are contained in $\frac{1}{I}\bN$, we see that the degree of each component of $D_i$ is upper bounded. Applying a standard argument on Chow scheme we see that $D_i$ lies in a bounded family, hence $\mE$ is log bounded. The rest of the proof is devoted to the boundedness of $\mG$.

\

\textit{Step 2}. In this step, we want to replace $(X, \sum_{i=1}^kD_i)\in \mE$ with $(X, \sum_{i=1}^kD_i')$ such that  $(X, \sum_{i=1}^kD_i')\in \mE$ and $(X, D_i')$ is a  log canonical Calabi-Yau pair for every $1\leq i\leq k$. By Theorem \ref{thm: complement}, there exists a positive integer $m(d)$ depending only on the dimension $d$ such that $X$ admits $m(d)$-complements. We may choose $m(d)$ sufficiently divisible such that $I| m(d)$. Put $\bP:=\frac{1}{m(d)}|-m(d)K_X|$. Then we see that $D_i\in \bP$ for every $1\leq i\leq k$. Since $(X, \sum_{i=1}^kD_i)\in \mE$, there exist some rational numbers $0\leq c_i<1$ such that $(X, \sum_i c_iD_i)$ is K-semistable. We are ready to replace $D_1$ with some $D_1'$.
Let $\mD\subset X\times \bP$ be the universal divisor associated to the linear system $\frac{1}{m(d)}|-m(d)K_X|$ and consider the universal family
$$ (X\times \bP, c_1\mD+\sum_{j=2}^kc_jD_j\times \bP)\to \bP .$$
Since there exists a fiber which is K-semistable, i.e. $(X, \sum_{i=1}^kc_iD_i$), one could find an open subset $U\subset \bP$ such that $(X, c_1\mD_t+\sum_{j=2}^kc_jD_j)$ is K-semistable for any $t\in U$ (see \cite{Xu20, BLX22}). Recall that $X$ admits $m(d)$-complements. Thus we conclude that $(X, \mD_t)$ is log canonical for general $t\in U$. Replacing $D_1$ with such a $\mD_t$, we obtain $D_1'$ as required. By the same way, we could replace other $D_j, 2\leq j\leq k$, step by step.

\

\textit{Step 3}. By step 2, we assume $(X, \sum_{i=1}^kD_i)\in \mE$ satisfies that $(X, D_i)$ is log canonical for every $1\leq i\leq k$. For such log pair $(X, \sum_iD_i)$, we define the following invariant:
$$\mu(X, \sum_{i=1}^kD_i):=\inf\Bigg\{\sum_{i=1}^kc_i\ |\ \textit{$(X, \sum_ic_iD_i)$ is K-semistable}\Bigg\}. $$
It is clear that $0\leq \mu(X, \sum_iD_i)<1$. We aim to show that there is a gap between $\mu(X, \sum_iD_i)$ and $1$. More precisely, there exists a positive real number $0<\epsilon_0(d,k,v, I)<1$ depending only on $d,k,v,I$ such that
$$1-\mu(X, \sum_{i=1}^kD_i)\geq \epsilon_0. $$
Suppose not, we could find a sequence of pairs $(X_j, \sum_{i=1}^kD_{ji})$ satisfying the following conditions:
\begin{enumerate}
\item $(X_j, \sum_{i=1}^kD_{ji})\in \mE$ for every $j$,
\item $(X_j, D_{ji})$ is log canonical for every $1\leq i\leq k$ and $j$,
\item $\mu_j:=\mu(X_j, \sum_{i=1}^kD_{ji})$ is an increasing sequence tending to $1$.
\end{enumerate}
By the definition of $\mu_j$, one could find an increasing sequence of rational numbers $a_j<\mu_j$ tending to $1$ such that
$$(X_j, \sum_{i=1}^k\frac{a_j}{k}D_{ji}) $$
is K-unstable for every $j$. By \cite{BLZ22, LXZ22}, for each $(X_j, \sum_{i=1}^k\frac{a_j}{k}D_{ji}) $, there exists a prime divisor $E_j$ over $X_j$ such that $E_j$ computes the delta invariant of $(X_j, \sum_{i=1}^k\frac{a_j}{k}D_{ji})$ and $E_j$ induces a special test configuration 
$$(\mX_j, \sum_{i=1}^k\frac{a_j}{k}\mD_{ji})\to \bA^1.$$
Subtracting those $a_j$ which are not close enough to $1$ and applying ACC of log canonical thresholds (e.g. \cite{HMX14}), we may assume that   $(\mX_{j, 0}, \sum_{i=1}^k\frac{1}{k}\mD_{ji, 0})$ is log canonical. This means that the test configuration degenerates the log canonical Calabi-Yau pair $(X_j, \sum_{i=1}^k\frac{1}{k}D_{ji})$ to another log canonical Calabi-Yau pair $(\mX_{j, 0}, \sum_{i=1}^k\frac{1}{k}\mD_{ji, 0})$.
By \cite[Lemma 2.8]{Zhou23}, $E_j$ is an lc place of $(X_j, \sum_{i=1}^k\frac{1}{k}D_{ji})$, which forces that $E_j$ is an lc place of $(X_j, D_{ji})$ for every $1\leq i\leq k$. Recalling that $E_j$ computes the delta invariant of the K-unstable log Fano pair $(X_j, \sum_{i=1}^k\frac{a_j}{k}D_{ji})$, we have
\begin{align*}
\beta_{(X_j, \sum_{i=1}^k\frac{a_j}{k}D_{ji})}(E_j)\ &=\ A_{(X_j, \sum_{i=1}^k\frac{a_j}{k}D_{ji})}(E_j) -S_{(X_j, \sum_{i=1}^k\frac{a_j}{k}D_{ji})}(E_j)\\
&=\ (1-a_j)\{A_{X_j}(E_j)-S_{X_j}(E_j)\}\\
&<\ 0.
\end{align*}
On the other hand, $(X_j, \sum_{i=1}^kc_{ji}D_{ji})$ is K-semistable log Fano for some rational $0\leq c_{ji}<1$, thus $0\leq \sum_{i=1}^kc_{ji}<1$ and
\begin{align*}
\beta_{(X_j, \sum_{i=1}^kc_{ji}D_{ji})}(E_j)\ &=\ A_{(X_j, \sum_{i=1}^kc_{ji}D_{ji})}(E_j) -S_{(X_j, \sum_{i=1}^kc_{ji}D_{ji})}(E_j)\\
&=\ (1-\sum_{i=1}^kc_{ji})\{A_{X_j}(E_j)-S_{X_j}(E_j)\}\\
&\geq\ 0,
\end{align*}
which is a contradiction. This contradiction implies the existence of the gap we want.

\

\textit{Step 4}. Combining step 2 and step 3, we see that for each $(X, \sum_{i=1}^kD_i)\in \mE$, one can find another log pair $(X, \sum_{i=1}^kD_i')\in \mE$ such that $(X, \sum_ic_iD_i')$ is K-semistable for some numbers $0\leq c_i<1$ with
$$1-\sum_{i=1}^kc_i\geq \epsilon_0(d, k,v, I). $$
Thus we have
\begin{align*}
\frac{A_{X}(E)}{(1-\sum_ic_i)S_X(E)}\ &\geq  \ \frac{A_{(X, \sum_{i=1}^kc_iD_i)}(E)}{S_{(X, \sum_{i=1}^kc_iD_i)}(E)}\geq 1
\end{align*}
for any prime divisor $E$ over $X$. Hence,
$$\frac{A_X(E)}{S_X(E)}\geq 1-\sum_{i=1}^kc_i\geq \epsilon_0 $$
for any prime divisor $E$ over $X$. This says that the delta invariant of $X$ is bounded from below by $\epsilon_0$. By Theorem \ref{thm: jiang} and Remark \ref{rmk: remarks} (2), the set $\mG$ defined in step 1 lies in a bounded family. The proof is finished.
\end{proof}

\section{On K-semistable domains}

We have studied the boundedness of two kinds of log Fano pairs with certain K-stability. It is then natural to ask what are the K-semistable domains for these log Fano pairs. We present some examples to predict that the domains should be polytopes for log pairs in $\mE$, and we will explore this problem from a theoretical viewpoint in the future work.

\begin{definition}\label{def: kss domain}
Let $(X, \sum_{i=1}^kD_i)$ be a log Fano manifold or a log pair in the set $\mE$. We define the \emph{K-semistable domain} of $(X, \sum_{i=1}^kD_i)$ as follows:
$$\mathrm{Kss}(X, \sum_{i=1}^kD_i):=\overline{\{(c_1,...,c_k)\ |\ \textit{$c_i\in [0,1)\cap \bQ$ and $(X, \sum_{i=1}^kc_iD_i)$ is K-semistable}\}}. $$
The overline in the definition means taking the closure.
\end{definition}

Before presenting the examples, let us first note the following interpolation property for K-stability (e.g. \cite[Proposition 2.13]{ADL24}), which we will use frequently: if $(X, \Delta_1)$ and $(X, \Delta_2)$ are both K-semistable log pairs (log Fano or log Calabi-Yau), where $\Delta_i$ are propotional to $-K_X$, then $(X, t\Delta_1+(1-t)\Delta_2)$ is also K-semistable for any $t\in [0,1]\cap \bQ$.

\begin{example}
Consider a log pair $(X, D_1+D_2):=(\bP^2, 3L_1+3L_2)$, where $L_1, L_2$ are two distinct lines in $\bP^2$. Then 
$$\mathrm{Kss}(X, D_1+D_2)=\{(0,0)\}.$$
To see this, we suppose $(x, y)\in \mathrm{Kss}(X, D_1+D_2)$, then $(\bP^2, 3xL_1+3yL_2)$ is K-semistable. Applying the beta-invariant criterion for K-semistability, we have
\begin{align*}
\beta_{(\bP^2, 3xL_1+3yL_2)}(L_1)\ &=\ 1-3x-\frac{1}{(3-3x-3y)^2}\int_0^{3-3x-3y}(3-3x-3y-t)^2\dt \\
&=\ 1-3x-(1-x-y)\\
&=\ (x+y)-3x\geq 0.
\end{align*}
Similarly, 
\begin{align*}
\beta_{(\bP^2, 3xL_1+3yL_2)}(L_2)\ 
&=\ 1-3y-\frac{1}{(3-3x-3y)^2}\int_0^{3-3x-3y}(3-3x-3y-t)^2\dt \\
&=\ 1-3y-(1-x-y)\\
&=\ (x+y)-3y\geq 0.
\end{align*}
The above two equations have the unique solution $(0,0)$, which agrees with Theorem \ref{thm-2-dim}. 
\end{example}

\begin{example}
Consider a log smooth pair $(X, D_1+D_2):=(\bP^2, \frac{3}{2}Q_1+\frac{3}{2}Q_2)$, where $Q_1, Q_2$ are two distinct smooth conics in $\bP^2$. We show that  $\mathrm{Kss}(X, D_1+D_2)$
is given by the following polytope:
\begin{equation*}
\begin{cases}
x\geq 0\\
y\geq 0\\
y\leq \frac{1}{2}x+\frac{1}{2}\\
y\geq 2x-1\\
x+y\leq 1\\
\end{cases}
\end{equation*}
It is presented by the following picture:

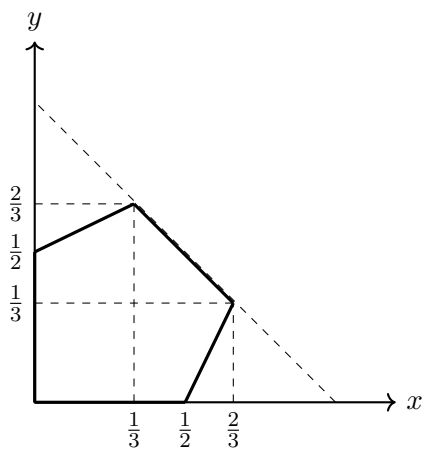
\begin{center}
\begin{tikzpicture}[scale=4]
    \coordinate (a) at (0,0);
    \coordinate (x1) at (0.33,0); \node[below] () at (x1) {$\frac{1}{3}$};
    \coordinate (x2)at (0.5,0);\node[below] () at (x2) {$\frac{1}{2}$};
    \coordinate (x3) at (0.66,0);\node[below] () at (x3) {$\frac{2}{3}$};
    \coordinate (x4) at (1,0);
    \coordinate (y1) at (0,0.33);\node[left] () at (y1) {$\frac{1}{3}$};
    \coordinate (y2)at (0,0.5);\node[left] () at (y2) {$\frac{1}{2}$};
    \coordinate (y3) at (0,0.66);\node[left] () at (y3) {$\frac{2}{3}$};
    \coordinate (y4) at (0,1);
    \coordinate (z13) at (0.33,0.66);
    \coordinate (z31) at (0.66,0.33);

    \draw [<->,thick] (0,1.2) node (yaxis) [above] {$y$}
        |- (1.2,0) node (xaxis) [right] {$x$};
    \draw[dashed] (x4) -- (y4);
    \draw[dashed] (x1) --(z13);
    \draw[dashed] (x3) -- (z31);
    \draw[dashed] (y3) --(z13);
    \draw[dashed] (y1) -- (z31);
    \draw[very thick] (a)--(y2); \draw[very thick] (y2) -- (z13); \draw[very thick] (z13)-- (z31); \draw[very thick] (z31)-- (x2); \draw[very thick] (x2) -- (a);
  \end{tikzpicture}
  \end{center}
  The extremal points of this polytope correspond to the pairs 
$$\bP^2,\ \ \ (\bP^2, \frac{3}{4}Q_2),\ \ \ (\bP^2, \frac{1}{2}Q_1+Q_2),\ \ \ (\bP^2, Q_1+\frac{1}{2}Q_2),\ \ \ (\bP^2, \frac{3}{4}Q_1),$$ 
all of which are K-semistable.   Note that $(\bP^2, \frac{1}{2}Q_1+Q_2)$ and $(\bP^2, Q_1+\frac{1}{2}Q_2)$ are K-semistable since they are log canonical Calabi-Yau pairs, and $(\bP^2, \frac{3}{4}Q_1)$ and $(\bP^2, \frac{3}{4}Q_2)$ are K-semistable by \cite[Theorem 1.5]{LS14} (see also Lemma \ref{lem: zz}). Thus the polytope is contained in $\mathrm{Kss}(X, D_1+D_2)$ by the interpolation property of K-stability. To see that it is exactly the K-semistable domain, it is enough to show that the points in 
\begin{equation*}
\begin{cases}
0< x< \frac{1}{3}\\
\frac{1}{2}<y<\frac{2}{3}\\
y>\frac{1}{2}x+\frac{1}{2}\\
\end{cases}
\
\
\text{and}
\
\
\
\
\
\begin{cases}
\frac{1}{2}<x<\frac{2}{3}\\
0<y<\frac{1}{3}\\
y<2x-1\\
\end{cases}
\end{equation*}
are not K-semistable. These two domains are presented below:

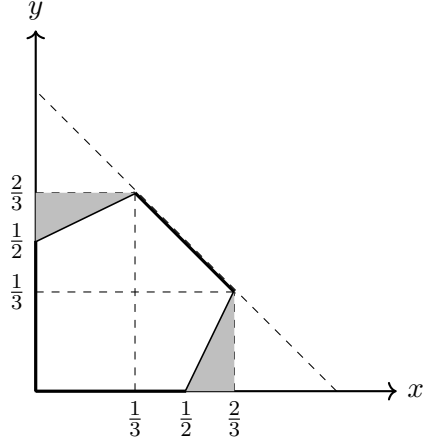
\begin{center}
\begin{tikzpicture}[scale=4]

    \coordinate (a) at (0,0);
    \coordinate (x1) at (0.33,0); \node[below] () at (x1) {$\frac{1}{3}$};
    \coordinate (x2)at (0.5,0);\node[below] () at (x2) {$\frac{1}{2}$};
    \coordinate (x3) at (0.66,0);\node[below] () at (x3) {$\frac{2}{3}$};
    \coordinate (x4) at (1,0);
    \coordinate (y1) at (0,0.33);\node[left] () at (y1) {$\frac{1}{3}$};
    \coordinate (y2)at (0,0.5);\node[left] () at (y2) {$\frac{1}{2}$};
    \coordinate (y3) at (0,0.66);\node[left] () at (y3) {$\frac{2}{3}$};
    \coordinate (y4) at (0,1);
    \coordinate (z13) at (0.33,0.66);
    \coordinate (z31) at (0.66,0.33);

    \draw [<->,thick] (0,1.2) node (yaxis) [above] {$y$}
        |- (1.2,0) node (xaxis) [right] {$x$};
    \draw[dashed] (x4) -- (y4);
    \draw[dashed] (x1) --(z13);
    \draw[dashed] (x3) -- (z31);
    \draw[dashed] (y3) --(z13);
    \draw[dashed] (y1) -- (z31);
    \draw[very thick] (a)--(y2); \draw[very thick] (y2) -- (z13); \draw[very thick] (z13)-- (z31); \draw[very thick] (z31)-- (x2); \draw[very thick] (x2) -- (a);

    \path[fill=lightgray] (y2)--(z13)--(y3)--(y2);
    \path[fill=lightgray] (x2)--(z31)--(x3)--(x2);

\end{tikzpicture}
\end{center}

Suppose $(x,y)\in \mathrm{Kss}(X, D_1+D_2)$, then $(\bP^2, \frac{3}{2}xQ_1+\frac{3}{2}yQ_2)$ is K-semistable. Applying the beta-invariant criterion for K-semistability, we have
\begin{align*}
\beta_{(\bP^2, \frac{3}{2}xQ_1+\frac{3}{2}yQ_2)}(Q_1)\ &=\ 1-\frac{3}{2}x-\frac{1}{(3-3x-3y)^2}\int_0^\frac{{3-3x-3y}}{2}(3-3x-3y-2t)^2\dt \\
&=\ 1-\frac{3}{2}x-\frac{1-x-y}{2}\\
&=\ \frac{1}{2}y-x+\frac{1}{2}\geq 0.
\end{align*}
Similarly, 
\begin{align*}
\beta_{(\bP^2, \frac{3}{2}xQ_1+\frac{3}{2}yQ_2)}(Q_2)\ &=\ 1-\frac{3}{2}y-\frac{1}{(3-3x-3y)^2}\int_0^\frac{{3-3x-3y}}{2}(3-3x-3y-2t)^2\dt \\
&=\ 1-\frac{3}{2}y-\frac{1-x-y}{2}\\
&=\ -y+\frac{1}{2}x+\frac{1}{2}\geq 0.
\end{align*}
It is clear that the points of the two shadowed domains cannot satisfy the above two equations, which implies that they are not contained in the K-semistable domain.
\end{example}

\begin{example}
Consider a log pair $(X, D_1+D_2):=(\bP^2, \frac{3}{2}Q+3L)$, where $Q$ is a smooth conic and $L$ is a line such that $(\bP^2, Q+L)$ is log smooth. We show that $\mathrm{Kss}(X, D_1+D_2)$ is the polytope generated by the points
$$(0,0),\ \ \ (\frac{1}{2},0),\ \ \ (\frac{2}{3}, \frac{1}{3}).$$ 
It is clear that the three points correspond to the the three log pairs 
$$\bP^2,\ \ \ (\bP^2, \frac{3}{4}Q),\ \ \ (\bP^2, Q+L).$$ 
We denote this polytope by $P$. To see $\mathrm{Kss}(X, D_1+D_2)=P$,  first note that the three pairs are all K-semistable, thus $P\subset \mathrm{Kss}(X, D_1+D_2)$. To see the converse inclusion, we apply the beta-invariant criterion for K-semistability to $Q$, $L$. Suppose $(x,y)\in \mathrm{Kss}(X, D_1+D_2)$, then $(\bP^2, \frac{3}{2}xQ+3yL)$ is K-semistable. Thus

\begin{align*}
\beta_{(\bP^2, \frac{3}{2}xQ+3yL)}(Q)\ &=\ 1-\frac{3}{2}x-\frac{1}{(3-3x-3y)^2}\int_0^\frac{{3-3x-3y}}{2}(3-3x-3y-2t)^2\dt \\
&=\ 1-\frac{3}{2}x-\frac{1-x-y}{2}\\
&=\ \frac{1}{2}y-x+\frac{1}{2}\geq 0,
\end{align*}

\begin{align*}
\beta_{(\bP^2, \frac{3}{2}xQ+3yL)}(L)\ &=\ 1-3y-\frac{1}{(3-3x-3y)^2}\int_0^{3-3x-3y}(3-3x-3y-t)^2\dt \\
&=\ 1-3y-(1-x-y)\\
&=\ x-2y\geq 0.
\end{align*}

The polytope defined by 
\begin{equation*}
\begin{cases}
0\leq x\leq \frac{2}{3}\\
0\leq y\leq \frac{1}{3}\\
\frac{1}{2}y-x+\frac{1}{2}\geq 0\\
x-2y\geq 0\\
\end{cases}
\end{equation*}
is exactly $P$, thus $P=\mathrm{Kss}(X, D_1+D_2)$.
\end{example}

We now plan to work out a class of examples in higher dimensions. Before that, we first list some results we will use.

Let $V$ be a projective Fano manifold of dimension $n$, and $S$ is a smooth divisor on $V$ such that $S\sim_\bQ -\lambda K_V$ for some positive rational number $\lambda$. Recall that
$$\mathrm{Kss}(V,S)=\overline{\{a\in [0,1)\cap \bQ\ |\ \textit{$(V,aS)$ is K-semistable} \}}. $$

\begin{lemma}{\rm{(\cite[Theorem 5.1]{ZZ22})}}\label{lem: zz}
Notation as above, suppose V and S are both K-semistable and $0<\lambda<1$, then $\mathrm{Kss}(V,S)= [0,1-\frac{r}{n}]$, where $r=\frac{1}{\lambda}-1$.
\end{lemma}

As a special case, we have

\begin{lemma}{\rm{(\cite{ZZ22})}}
Let $(\bP^n, S_d)$ be a log pair where $S_d$ is a hypersurface of degree $1\leq d\leq n$. Suppose $S_d$ is K-semistable (this is known for general smooth hypersurfaces and expected to be true for all smooth hypersurfaces). Then we have $\mathrm{Kss}(\bP^n, S_d)=[0, 1-\frac{r}{n}]$, where $r=\frac{n+1-d}{d}$.
\end{lemma}

The following lemma is also well known, see e.g. \cite[Prop 2.11]{LZ22}.
\begin{lemma}\label{lem:cone stability}
Let $(V,\Delta)$ be an n-dimensional log Fano pair, and L an ample line bundle on V such that $L\sim_\bQ -\frac{1}{r}(K_V+\Delta)$ for some $0<r\leq n+1$. Suppose Y is the projective cone over V associated to L with infinite divisor $V_\infty$, then $(V,\Delta)$ is K-semistable  if and only if $(Y,\Delta_Y+(1-\frac{r}{n+1})V_\infty)$ is K-semistable, where $\Delta_Y$ is the divisor on Y naturally extended by $\Delta$ (i.e. $\Delta_Y$ is the cone over $\Delta$).
\end{lemma}

Applying Lemma \ref{lem: zz}, a simple computation tells us the following three facts:

\begin{enumerate}
\item Let $Q\subset \bP^n$ be a smooth quadric, then $(\bP^n, aQ)$ is K-semistable for any $a\in [0, \frac{n+1}{2n}]$. Moreover, $(\bP^n, \frac{n+1}{2n}Q)$ is K-semistable and it admits a K-semistable degeneration $(X, \frac{n+1}{2n}D)$, where $X$ is a hypersurface in $\bP(1^{n+1}, 2)$ (with coordinates $(x_0,...,x_n, z)$) defined by $x_0^2+x_1^2+...+x_n^2=0$, and $D=X\cap \{z=0\}$.
\item Let $Q_l\subset \bP^{l+1}$ be a smooth quadric of dimension $l$, then $(Q_l, aQ_{l-1})$ is K-semistable  for any $a\in [0, \frac{1}{l}]$.
\item Let $Q_l, Q_l'\subset \bP^{l+1}$ ($l\geq 2$) be two distinct smooth quadric hypersurfaces such that $(\bP^{l+1}, Q_l+Q_l')$ is log smooth, then $(Q_l, a{Q'_{l}}|_{Q_l})$ is K-semistable  for any $a\in [0, \frac{l+2}{2l}]$ (note that $Q_l'\cap Q_l$ is K-semistable by \cite{AGP06}).
\end{enumerate}

We are ready to prove the following K-semistability.
\begin{theorem}\label{thm: polytope example}
For $n\geq 2$, the log pair $(\bP^n, \frac{n}{2(n-1)}Q+\frac{1}{n-1}L)$ is K-semistable, where $Q$ is a smooth quadric hypersurface and $L$ is a hyperplane such that $(\bP^n, Q+L)$ is log smooth.
\end{theorem}

\begin{proof}
The case $n=2$ is clear. We assume $n\geq 3$. As we have observed,  $(\bP^n, \frac{n+1}{2n}Q)$ is K-semistable and admits a K-semistable degeneration $(X, \frac{n+1}{2n}D)$,  where $X$ is a hypersurface in $\bP(1^{n+1}, 2)$ (with coordinates $(x_0,...,x_n, z)$) defined by $x_0^2+x_1^2+...+x_n^2=0$, and $D=X\cap \{z=0\}$. Denote by $D'$ the corresponding degeneration of $L$ under the test configuration. 

It suffices to show that the log pair $(X, \frac{n}{2(n-1)}D+\frac{1}{n-1}D')$ is a K-semistable log Fano pair.  To see this, first note that 
$(X, \frac{n}{2(n-1)}D+\frac{1}{n-1}D')$ is the projective cone over $(Q, \frac{1}{n-1}L\cap Q)$ with respect to the polarization $\mO_Q(2):=i^*\mO_{\bP^n}(2)$, where $i: Q\to \bP^n$ is the natural embedding.  We have the following computation:
$$\mO_Q(2)\sim_\bQ -\frac{1}{r}(K_{Q}+\frac{1}{n-1}L\cap Q) \quad \text{and} \quad  \frac{n}{2(n-1)}=1-\frac{r}{n}$$
for 
$$r= \frac{n-1-\frac{1}{n-1}}{2}.$$ 
By our notation, we have
$$(Q_{n-1}, \frac{1}{n-1}Q_{n-2})=(Q, \frac{1}{n-1}L\cap Q),$$
which is K-semistable as we have seen before.  By  Lemma \ref{lem:cone stability}, we see that $(X, \frac{n}{2(n-1)}D+\frac{1}{n-1}D')$ is K-semistable.
\end{proof}

\begin{theorem}\label{thm: polytope example II}
For $n\geq 3$, the log pair $(\bP^n, \frac{n+1}{2(n-1)}Q+\frac{n+1}{2(n-1)}Q')$ is K-semistable, where $Q, Q'$ are smooth quadric hypersurfaces  such that $(\bP^n, Q+Q')$ is log smooth.
\end{theorem}

\begin{proof}
The case $n=3$ is clear. We assume $n\geq 4$. The same as before,  $(\bP^n, \frac{n+1}{2n}Q)$ is K-semistable and admits a K-semistable degeneration $(X, \frac{n+1}{2n}D)$,  where $X$ is a hypersurface in $\bP(1^{n+1}, 2)$ (with coordinates $(x_0,...,x_n, z)$) defined by $x_0^2+x_1^2+...+x_n^2=0$, and $D=X\cap \{z=0\}$. Denote by $D'$ the corresponding degeneration of $Q'$ under the test configuration. 

It suffices to show that the log pair $(X, \frac{n+1}{2(n-1)}D+\frac{n+1}{2(n-1)}D')$ is a K-semistable log Fano pair.  To see this, first note that 
$(X, \frac{n+1}{2(n-1)}D+\frac{n+1}{2(n-1)}D')$ is the projective cone over $(Q, \frac{n+1}{2(n-1)}Q'|_Q)$ with respect to the polarization $\mO_Q(2):=i^*\mO_{\bP^n}(2)$, where $i: Q\to \bP^n$ is the natural embedding.  We have the following computation:
$$\mO_Q(2)\sim_\bQ -\frac{1}{r}(K_{Q}+\frac{n+1}{2(n-1)}Q'|_Q) \quad \text{and} \quad  \frac{n+1}{2(n-1)}=1-\frac{r}{n}$$
for 
$$r= \frac{n-1-\frac{n+1}{n-1}}{2}.$$ 
Note that $(Q, \frac{n+1}{2(n-1)}Q'|_Q)$ is K-semistable.  By  Lemma \ref{lem:cone stability}, we see that $(X, \frac{n+1}{2(n-1)}D+\frac{n+1}{2(n-1)}D')$ is K-semistable.
\end{proof}

Armed by Theorem \ref{thm: polytope example},  \ref{thm: polytope example II} , we consider the following two classes of examples.

\begin{example}\label{exa: polytope example}
Consider a pair $(\bP^n, Q+L)$ for $n\geq 2$, where $Q$ is a smooth quadric hypersurface and $L$ is a hyperplane such that $(\bP^n, Q+L)$ is log smooth. We want to compute $\mathrm{Kss}(\bP^n, Q+L)$.

Suppose $(x,y)\in \mathrm{Kss}(\bP^n, Q+L)$, then $(\bP^n, xQ+yL)$ is K-semistable. Applying the beta-invariant criterion for K-semistability, we have
\begin{align*}
\beta_{(\bP^n, xQ+yL)}(Q)\ &=\ 1-x-\frac{1}{(n+1-2x-y)^n}\int_0^\frac{{n+1-2x-y}}{2}(n+1-2x-y-2t)^n\dt \\
&=\ 1-x-\frac{n+1-2x-y}{2(n+1)}\\
&=\ \frac{1}{2}-\frac{nx}{n+1}+\frac{y}{2(n+1)}\geq 0.
\end{align*}
Similarly, 
\begin{align*}
\beta_{(\bP^n, xQ+yL)}(L)\ &=\ 1-y-\frac{1}{(n+1-2x-y)^n}\int_0^{n+1-2x-y}(n+1-2x-y-t)^n\dt \\
&=\ 1-y-\frac{n+1-2x-y}{(n+1)}\\
&=\ \frac{2x}{n+1}-\frac{ny}{n+1}\geq 0.
\end{align*}
It is clear that the polytope given by
\begin{equation*}
\begin{cases}
0\leq x\leq 1\\
0\leq y\leq 1\\
\frac{1}{2}-\frac{nx}{n+1}+\frac{y}{2(n+1)}\geq 0\\
\frac{2x}{n+1}-\frac{ny}{n+1}\geq 0\\
\end{cases}
\end{equation*}
is generated by the extremal points 
$$(0,0),\ \ \ (\frac{n+1}{2n}, 0),\ \ \ (\frac{n}{2(n-1)},\frac{1}{n-1}).$$
 These three points correspond to log pairs 
$$\bP^n, \ \ \ (\bP^n, \frac{n+1}{2n}Q), \ \ \ (\bP^n, \frac{n}{2(n-1)}Q+\frac{1}{n-1}L),$$ which are all K-semistable by Theorem \ref{thm: polytope example}. Thus the polytope is exactly $\mathrm{Kss}(\bP^n, Q+L)$. We also mention that when $n=3$, the example corresponds to  \cite[Theorem 1 (1)]{Loginov21}.
\end{example}

\begin{example}\label{exa: polytope example II}
Consider a log pair $(\bP^n, Q+Q')$ for $n\geq 3$, where $Q, Q'$ are smooth quadric hypersurfaces  such that $(\bP^n, Q+Q')$ is log smooth. We want to compute $\mathrm{Kss}(\bP^n, Q+Q')$.

Suppose $(x,y)\in \mathrm{Kss}(\bP^n, Q+Q')$, then $(\bP^n, xQ+yQ')$ is K-semistable. Applying the beta-invariant criterion for K-semistability, we have
\begin{align*}
\beta_{(\bP^n, xQ+yQ')}(Q)\ &=\ 1-x-\frac{1}{(n+1-2x-2y)^n}\int_0^\frac{{n+1-2x-2y}}{2}(n+1-2x-2y-2t)^n\dt \\
&=\ 1-x-\frac{n+1-2x-2y}{2(n+1)}\\
&=\ \frac{1}{2}-\frac{nx}{n+1}+\frac{y}{n+1}\geq 0.
\end{align*}
Similarly, 
\begin{align*}
\beta_{(\bP^n, xQ+yQ')}(Q')\ &=\ 1-y-\frac{1}{(n+1-2x-2y)^n}\int_0^\frac{{n+1-2x-2y}}{2}(n+1-2x-2y-2t)^n\dt \\
&=\ 1-y-\frac{n+1-2x-2y}{2(n+1)}\\
&=\ \frac{1}{2}-\frac{ny}{n+1}+\frac{x}{n+1}\geq 0.
\end{align*}
It is clear that the polytope given by
\begin{equation*}
\begin{cases}
0\leq x\leq 1\\
0\leq y\leq 1\\
\frac{1}{2}-\frac{nx}{n+1}+\frac{y}{n+1}\geq 0\\
\frac{1}{2}-\frac{ny}{n+1}+\frac{x}{n+1}\geq 0\\
\end{cases}
\end{equation*}
is generated by the extremal points 
$$(0,0),\ \ \ (\frac{n+1}{2n}, 0),\ \ \ (0, \frac{n+1}{2n}),\ \ \ (\frac{n+1}{2(n-1)},\frac{n+1}{2(n-1)}).$$ 
These four points correspond to log pairs 
$$\bP^n,\ \ \ (\bP^n, \frac{n+1}{2n}Q),\ \ \ (\bP^n, \frac{n+1}{2n}Q'),\ \ \ (\bP^n, \frac{n+1}{2(n-1)}Q+\frac{n+1}{2(n-1)}Q'), $$ which are all K-semistable by Theorem \ref{thm: polytope example II}. Thus the polytope is exactly $\mathrm{Kss}(\bP^n, Q+Q')$.
\end{example}

\begin{remark}
For all examples we treat above, the K-semistable domains are polytopes. In \cite{Loginov21}, the K-semistable domains for some log Fano pairs of Maeda type in dimension three are computed, but it is hard to say whether these sets are convex or polytopes.

\end{remark}

\bibliography{reference.bib}
\end{document}